\documentclass[a4paper, 12pt, twoside]{article}

\title{Girsanov's formula for $G$-Brownian motion}
\author{Emi Osuka\thanks{Mathematical Institute, 
Tohoku University, Aoba-ku, Sendai 980-8578, Japan.}}
\date{\empty}

\usepackage{amsmath,mathtools,amsthm,enumerate,comment,color}
\usepackage{amssymb,mathrsfs,bm,latexsym}
\usepackage{fancyhdr}

\allowdisplaybreaks
\numberwithin{equation}{section}
\mathtoolsset{showonlyrefs=true}

\theoremstyle{plain}
	\newtheorem{thm}{Theorem}[section]
	\newtheorem{prop}[thm]{Proposition}
	\newtheorem{lem}[thm]{Lemma}
	
\theoremstyle{definition}
	\newtheorem{defi}[thm]{Definition}
\theoremstyle{remark}
	\newtheorem{rem}[thm]{Remark}
	
\newcommand	\tref	{Theorem~\ref}
\newcommand	\pref	{Proposition~\ref}
\newcommand	\lref	{Lemma~\ref}
\newcommand	\cref	{Corollary~\ref}
\newcommand	\dref	{Definition~\ref}
\newcommand	\rref	{Remark~\ref}

\newcommand	\E	{\mathbb{E}}		\newcommand	\F		{\mathbb{F}}

		\newcommand	\N	{\mathbb{N}}
		
		\newcommand	\R	{\mathbb{R}}

\newcommand	\calA	{\mathcal{A}}		\newcommand	\calB	{\mathcal{B}}
		
		\newcommand	\calF		{\mathcal{F}}
		\newcommand	\calH		{\mathcal{H}}
		
		\newcommand	\calL		{\mathcal{L}}
		\newcommand	\calN		{\mathcal{N}}

	\newcommand	\scrL	{\mathscr{L}}
	
	\newcommand	\scrP	{\mathscr{P}}

\newcommand	\qv	[1]	{\langle #1 \rangle}
\newcommand	\fp	[2]	{\frac{\partial #1}{\partial #2}}
\newcommand	\esp	[1]	{\underset{#1}{\mathrm{ess\,sup\,}}}

\newcommand	\bt		{\textbf}

\newcommand	\ddd		{, \dots ,}
\newcommand	\n			{\text}
\newcommand	\one		{\normalfont{\mbox{1}\hspace{-0.25em}\mbox{l}}}

\newcommand	\qd		{\quad}
\newcommand	\qn		{\qd\text}

\newcommand	\tr		{\mathrm{tr}}

\newcommand	\ve		{\varepsilon}
\newcommand	\vp		{\varphi}

\renewcommand	\ge	{\geqslant}
\renewcommand	\le	{\leqslant}

\newcommand	\ol	{\bar}

\newcommand	\bLip	{C_{b,Lip}}
\newcommand	\lLip	{C_{l,Lip}}
\newcommand	\cA	{\calA^{\Theta}_{0,T}}

\newcommand	\tp	{\theta^{\prime}}

\newcommand	\tpB	[2]{B^{#1 ,\tp}_{#2}}
\newcommand	\tB	[2]{B^{#1 , \theta}_{#2}}

\begin{document}
\maketitle
\thispagestyle{empty}
\begin{abstract}
In this paper, we establish Girsanov's formula for $G$-Brownian motion. 
Peng (2007, 2008) constructed $G$-Brownian motion 
on the space of continuous paths under a sublinear expectation called $G$-expectation; 
as obtained by Denis et al.\ (2011), 
$G$-expectation is represented as the supremum of linear expectations 
with respect to martingale measures of a certain class. 
Our argument is based on this representation 
with an enlargement of the associated class of martingale measures, 
and on Girsanov's formula for martingales in the classical stochastic analysis. 
The methodology differs from that of Xu et al.\ (2011), 
and applies to the multidimensional $G$-Brownian motion. 
\footnote{e-mail: sa9m06@math.tohoku.ac.jp}
\footnote{phone: +81~227956401, fax: +81~227956400}
\footnote{\bt{Key words}: 
$G$-Brownian motion, $G$-expectation, sublinear expectation space, 
Girsanov's formula, upper expectation.}
\end{abstract}
\section{Introduction}
Motivated by risk measures and volatility uncertainty problems in finance, 
S.~Peng introduced the notion of $G$-Brownian motion. 
Intuitively, $G$-Brownian motion is a Brownian motion whose variance is uncertain. 
While the classical Brownian motion is defined on a probability space, 
$G$-Brownian motion is defined on a sublinear expectation space, that is, 
the triple $(\Omega , \calH , \E)$, where $\Omega$ is a given set 
and $\calH$ is a vector lattice of real-valued functions on $\Omega$ containing $1$, 
which is the domain of a sublinear expectation $\E$. 
$G$-Brownian motion is defined by using two notions 
concerning distributions on a sublinear expectation space: 
identical distributedness and independence. 
On a sublinear expectation space, the notion of distributions 
cannot be interpreted as that on a probability space; 
indeed, as introduced in \cite{Peng;10}, 
it also needs to be interpreted as a sublinear expectation on a class of test functions 
suitably chosen according to the domain $\calH$. 

Peng \cite{Peng;06,Peng;08b} constructed a sublinear expectation space 
on which the canonical process of the space 
$\Omega = C_0 ([0, \infty ) ; \R^d)$ of continuous paths starting from 0 
becomes a $G$-Brownian motion. 
The sublinear expectation in this space is called $G$-expectation. 
It\^o's integrals with respect to $G$-Brownian motion 
and the quadratic variation process of $G$-Brownian motion 
were also defined in \cite{Peng;06,Peng;08b}. 
Recently, L.~Denis, M.~Hu and S.~Peng proved in \cite{Denis;10} 
that $G$-expectation can be represented as the supremum of linear expectations, 
referred to as the upper expectation, with respect to martingale measures of a certain class. 

In this paper, we derive Girsanov's formula for $G$-Brownian motion; 
when we are given a $G$-Brownian motion and a drift 
on the sublinear expectation space of Peng \cite{Peng;06,Peng;08b}, 
we construct a new sublinear expectation space 
on which the $G$-Brownian motion with the drift is a $G$-Brownian motion. 
Through the construction, $G$-expectation is transformed into a weighted $G$-expectation. 
The weight  has the same form as that in the classical Girsanov's formula, 
in which It\^o's integral for $G$-Brownian motion and the quadratic variation process are involved. 
A remarkable point of the construction is that not only $G$-expectation but also its domain is changed. 
As a sublinear expectation space is the notion including the domain of a sublinear expectation, 
in general some care about the choice of domains is needed when changing sublinear expectations. 
In the course of our discussion, it is also required 
that the notion of distributions is appropriately defined in the new sublinear expectation space. 
Those are main reasons why the domain of $G$-expectation 
is changed in order to formulate Girsanov's formula for $G$-Brownian motion. 

In the classical stochastic analysis, Girsanov's formula for Brownian motion plays a fundamental role; 
it is applied in many directions such as the derivation of large deviations of Schilder's \cite{Schilder;66}, 
the construction of weak solutions to stochastic differential equations driven by Brownian motion and so on. 
Among them is the derivation of a variational representation 
for functionals of Brownian motion due to Bou\'e-Dupuis \cite{Boue;98}, 
where they also showed the usefulness of the representation 
by applying it to prove Laplace principles for families of functionals of Brownian motion. 
Using the main result of the present paper, 
we establish in \cite{Osuka;12} a variational representation for functionals of $G$-Brownian motion 
and show that a similar application is possible under the framework of $G$-expectation space. 
Independently of our work \cite{Osuka;12}, 
Gao \cite{Gao;12} also obtains the representation by using our Girsanov's formula, 
and discusses an application to a large deviation for stochastic flows driven by $G$-Brownian motion. 

The keys to the proof of our main result are: 
(i)~the representation of the upper expectation for $G$-expectation 
due to Denis-Hu-Peng \cite{Denis;10}, 
with an enlargement of the associated class of martingale measures
as given in Soner-Touzi-Zhang \cite{Soner;10a}; 
and (ii)~Girsanov's formula for martingales in the classical stochastic analysis. 
Our methodology is different from that of Xu-Shang-Zhang \cite{Xu;10}, 
in which they obtained Girsanov's formula for one-dimensional $G$-Brownian motion; 
their proof relies on the martingale characterization 
of one-dimensional $G$-Brownian motion in \cite{Xu;10a}, 
which restricts their argument to one dimension, 
whereas the method we employ 
in this paper equally works for multidimensional $G$-Brownian motion. 
See \rref{d:rem}. 

This paper is organized as follows. 
From Section \ref{Peng;10} through Section \ref{Denis;10}, 
we introduce necessary notions and related results as preliminaries: 
the notion of distributions on a sublinear expectation space, 
the construction of $G$-expectation, stochastic integrals for $G$-Brownian motion, 
and the upper expectation for $G$-expectation given by Denis-Hu-Peng \cite{Denis;10}. 
In Section \ref{Gir,section}, we state and prove Girsanov's formula for $G$-Brownian motion. 

\subsection{Notation}
\begin{itemize}
	\item $\bLip (\R^n)$ : 
		the space of all bounded and Lipschitz continuous functions on $\R^n$
	\item $\lLip ( \R^n )$ : 
		the space of all functions $\vp$ satisfying 
		\begin{align}
			| \vp (x) - \vp (y) |
			\le 
			C (1 + |x|^k +|y|^k) |x-y|
			\qn{for all } x,y \in \R^n
		\end{align}
		for some $C>0$, $k \in \N$ depending on $\vp$
	\item $\R^{d \times d}$ : all $d \times d$ real matrices
	\item $I_d$ : the $d \times d$ unit matrix
	\item $|x| := \sqrt{ x \cdot x}$ : 
		the norm of $x \in \R^n$, where $\cdot$ is the inner product of $\R^n$
	\item $\| A \| := \sqrt{ \tr [A A^*] }$ : 
		the norm of $A \in \R^{d \times d}$, where $A^*$ is the transposed matrix of $A$
	\item For a probability measure $P$, $E_P$ denotes the expectation with respect to $P$
\end{itemize}
In the sequel, unless otherwise stated, probability spaces 
we deal with are all assumed to be completed. 

\section{Sublinear expectation spaces}\label{Peng;10}

Following Peng \cite[Chapter~I]{Peng;10}, 
we introduce the definition of sublinear expectations and related notions.

Let $\Omega$ be a given set and $\calH$  a vector lattice 
of real functions on $\Omega$ containing $1$, 
that is, $\calH$ is a linear space such that $1 \in \calH$ 
and that $X \in \calH$ implies $|X| \in \calH$. 
\begin{defi}\label{slinear,def}
	A functional $\E : \calH \to \R$ is called a \bt{sublinear expectation} if it satisfies
	\begin{enumerate}[(i)]
		\item $\E [X] \le \E [Y] \qn{if } X \le Y$,
		\item $\E [c] = c \qn{for all } c \in \R$,
		\item $\E [X+Y] \le \E [X] + \E[Y] \qn{for all } X,Y \in \calH$,
		\item $\E [\lambda X] = \lambda \E[X] \qn{for all } \lambda \ge 0$.
	\end{enumerate}
	The triple $(\Omega , \calH , \E)$ is called a \bt{sublinear expectation space}.
\end{defi}

\begin{defi}
	Let $(\Omega ,\calH , \E)$ be a sublinear expectation space.
	$X= (X^1 ,\dots , X^n)$ is called an $n$-dimensional \bt{random vector}, 
	denoted by $ X \in \calH^n$, if $X^i \in \calH$ for each $i = 1\ddd n$.
	$\{ X_t ; t \ge 0 \}$ is called an $n$-dimensional \bt{stochastic process} 
	if for each $t \ge 0$, $X_t$ is an $n$-dimensional random vector.
\end{defi}

Next we introduce the notion of distributions of random variables under a sublinear expectation space. 
Let us consider the following sublinear expectation space: 
\begin{align}
	\n{for all } n \in \N \n{ and } \vp \in \lLip (\R^n) ,~
	X \in \calH^n \n{ implies } \vp (X) \in \calH .
	\label{closure}
\end{align}

\begin{defi}\label{dis}
	Let $X_1$ and $X_2$ be two $n$-dimensional random vectors,
	and $X_3$ an $m$-dimensional random vector defined 
	on a sublinear expectation space $(\Omega , \calH , \E)$.
	$X_1$ and $X_2$ are called \bt{identically distributed} if
	\begin{align}
		\E [ \vp (X_1) ] 
		= \E [\vp (X_2)] 
		\qn{for each } \vp \in \lLip ( \R^n ) .
		\label{id}
	\end{align}
	$X_3$ is said to be \bt{independent} from $X_1$ if
	\begin{align}
		\E[\vp (X_1 , X_3)] 
		= \E[ \E [\vp (x , X_3)] \big|_{x = X_1} ] 
		\qn{for each } \vp \in \lLip (\R^{n+m}) .
		\label{indep}
	\end{align}
\end{defi}

We remark that, as in \eqref{closure}, in order to define the notion of distributions, 
the essential requirement for $\calH$ is that $\calH$ is closed under 
substitutions of its elements into functions $\vp$ of a certain class, 
which may also be chosen, e.g., as $\bLip ( \R^n )$, 
the space of all bounded Lipschitz continuous functions on $\R^n$; 
then in the above definition, $\lLip ( \R^n )$ in \eqref{id} 
and $\lLip ( \R^{n+m} )$ in \eqref{indep} are replaced 
by $\bLip ( \R^n )$ and $\bLip ( \R^{n+m} )$, respectively.

\section{$\bm{G}$-Brownian motion and $\bm{G}$-expectation}\label{Peng;06-08}

Following Peng \cite{Peng;06, Peng;08b}, 
we introduce the construction of $G$-Brownian motion and related notions.

Throughout the paper, we fix $T>0$ and denote by $\Theta$ a fixed non-empty, 
bounded and closed subset of $\R^{d \times d}$.
Let $\Omega := C_0 ([0,T] ; \R^d)$ be the space of all $\R^d$-valued continuous functions 
$(\omega_t)_{t \in [0,T]}$ with $\omega_0 =0$, equipped with the distance
\begin{align}
	\rho (\omega^1 , \omega^2) := \max_{t \in [0,T]} | \omega^1_t - \omega^2_t | .
\end{align}
For each $t \in [0,T]$, we also set 
$\Omega_t := \{ \omega_{\cdot \wedge t} : \omega \in \Omega \}$.
We denote by $\calB (\Omega)$ (resp. $\calB (\Omega_t )$) 
the Borel $\sigma$-algebra on $\Omega$ (resp. $\Omega_t$).

\subsection{$\bm{G}$-Brownian motion and $\bm{G}$-expectation}

For each $\vp \in \bLip (\R^d)$, we denote by 
$u_{\vp} \in C([0,T] \times \R^d)$ the unique viscosity solution 
of the following nonlinear partial differential equation called $G$-heat equation:
\begin{align}\label{Geq}
	\left\{
	\begin{aligned}
		&\fp{u}{t} - G( D^2 u ) = 0 
		\qn{in } (0,T) \times \R^d , \\
		&u |_{t=0} = \vp 
		\qn{in } \R^d ,
	\end{aligned}
	\right.
\end{align}
where $D^2 u$ is the Hessian matrix of $u$ and
\begin{align}
	G (A) 
	:= 
	\sup_{\gamma \in \Theta} 
	\left\{ 
	\frac{1}{2} \tr [ \gamma \gamma^* A ] 
	\right\}
\end{align}
for a $d \times d$ symmetric real matrix $A$; 
for the existence and uniqueness of a viscosity solution of \eqref{Geq}, 
refer to Appendix C, Section 3 in \cite{Peng;10}.

\begin{rem}\label{hitaika}
	If there exists a constant $\sigma_0 >0$ such that 
	$\gamma \gamma^* \ge \sigma_0 I_d$ for all $\gamma \in \Theta$,
	then \eqref{Geq} has a unique $C^{1,2}$-solution.
\end{rem}

Let $B$ be the canonical process of $\Omega$.
For each $t \in [0,T]$, we denote by $\bLip (\Omega_t)$ 
the set of all bounded Lipschitz cylinder functionals on $\Omega_t$:
\begin{align}
	\bLip (\Omega_t) 
	:= 
	\{ \vp (B_{t_1} \ddd B_{t_n}) 
	: n \in \N ,~ t_1 \ddd t_n \in [0,t] ,~ \vp \in \bLip (( \R^d )^n) \} ,
\end{align}
and we write $\bLip (\Omega) \equiv \bLip (\Omega_T)$ simply.
We can construct a consistent sublinear expectation $\E$ on $\bLip (\Omega)$ such that
\begin{itemize}
	\item for all $0 \le s<t \le T$ and $\vp \in \bLip (\R^d)$,
		\begin{align}
			\E [\vp (B_t - B_s)] 
			= \E [\vp ( B_{t-s} )] 
			= u_{\vp} (t-s,0),
		\end{align}
	\item for all $n \in \N ,~0 \le t_1 < \dots < t_n \le T$ and $\vp \in \bLip ((\R^d)^n) $,
		\begin{align}
			\E [\vp (B_{t_1} \ddd B_{t_n})] 
			= \E [ \vp_1 (B_{t_1} \ddd B_{t_{n-1}})] ,
		\end{align}
		where $\vp_1 (x_1 \ddd x_{n-1}) 
		:= \E [\vp (x_1 \ddd x_{n-1} , B^{t_{n-1}}_{t_n} + x_{n-1}) ]$
		with $B^s_t := B_t -B_s$ for $0 \le s \le t \le T$.
\end{itemize}
For $t_{k-1} \le t < t_k$, the related conditional expectation of 
$\vp ( B_{t_1} \ddd B_{t_n} )$ on $\bLip (\Omega_t)$ is defined by
\begin{align}
	\E_t [ \vp ( B_{t_1} \ddd B_{t_n}) ] 
	:= \vp_{n-k} (B_{t_1} \ddd B_{t_{k-1}} , B_t),
\end{align}
where $\vp_{n-k} (x_1 \ddd x_{k-1} , x_k) 
= \E [\vp (x_1 \ddd x_{k-1 }, B^t_{t_k} +x_k \ddd B^t_{t_n} + x_k )]$.

Let $\calL^1_G (\Omega_t)$ be the completion of 
$\bLip (\Omega_t)$ under the norm $\E [| \cdot |]$,
and we write $\calL^1_G (\Omega) \equiv \calL^1_G (\Omega_T)$ simply.
We can extend $\E [ \cdot ]$ (resp. $\E_t [ \cdot ]$) 
to a unique sublinear expectation (resp. a conditional sublinear expectation) 
on $\calL^1_G (\Omega)$.
It is called $G$-expectation (resp. conditional $G$-expectation).

\begin{defi}\label{GBm,def}
	A stochastic process $B$ on $(\Omega , \calL^1_G (\Omega) , \E)$ 
	is called a $\bm{G}$\bt{-Brownian motion} if
	\begin{enumerate}[(i)]
		\item $B_0 =0$,
		\item for all $0 \le s<t \le T$ and $\vp \in \bLip (\R^d)$,
			\begin{align}
				\E [\vp (B_t - B_s)] 
				= \E [\vp ( B_{t-s} )] 
				= u_{\vp} (t-s,0),
			\end{align}
		\item for all $n \in \N ,~0 \le t_1 < \dots < t_n \le T$ and $\vp \in \bLip ((\R^d)^n) $,
			\begin{align}
				\E [\vp (B_{t_1} \ddd B_{t_n})] 
				= \E [ \vp_1 (B_{t_1} \ddd B_{t_{n-1}})] ,
			\end{align}
			where $\vp_1 (x_1 \ddd x_{n-1}) 
			:= \E [\vp (x_1 \ddd x_{n-1} , B_{t_n} - B_{t_{n-1}} + x_{n-1}) ]$.
	\end{enumerate}
\end{defi}

Note that (ii) means $B_t - B_s$ and $B_{t-s}$ are identically distributed,
and that (iii) means $B_{t_n} - B_{t_{n-1}}$ is independent from $(B_{t_1} \ddd B_{t_{n-1}})$.
From the above definition, we can see that 
on the sublinear expectation space $(\Omega , \calL^1_G (\Omega) , \E)$,
the canonical process is a $G$-Brownian motion.

\subsection{It\^o's integral for $\bm{G}$-Brownian motion}

For each $p \ge 1$, we denote by $\calL^p_G ( \Omega_t )$ 
the completion of $\bLip ( \Omega_t )$ under $\E[| \cdot |^p]^{1/p}$.
Let
\begin{align}
	M^{p,0}_G ( \Omega )
	:=
	\left\{
	\sum_{k=0}^{n-1} \xi_k \one_{[t_k , t_{k+1})} :
	n \in \N ,~ 0= t_0 < t_1 < \dots < t_{n} =T ,~ 
	\xi_k \in \calL^p_G ( \Omega_{t_k} )
	\right\},
\end{align}
and let $M^p_G ( \Omega )$ be the completion of $M^{p,0}_G ( \Omega )$ 
under $( \int_0^T \E [| \cdot |^p ] dt )^{1/p}$.

For every $h \in (M^2_G ( \Omega ))^d$, we denote
\begin{align}
	\int_0^t h_s \cdot d B_s
	= \sum_{i=1}^d \int_0^t h^i_s \, d B^i_s .
\end{align}
Here each summand denotes It\^o's integral 
with respect to the $i$-th coordinate $B^i$ of $G$-Brownian motion $B$,
which is defined as an element of $\calL^2_G ( \Omega_t )$.
For every $i,j = 1 \ddd d$, the mutual variation of $B^i$ and $B^j$
\begin{align}
	\qv{B^i , B^j}_t
	:= 
	B^i_t B^j_t 
	- \int_0^t B^i_s \, dB^j_s 
	- \int_0^t B^j_s \, dB^i_s
\end{align}
is also defined since $B^i , B^j \in M^2_G ( \Omega )$.
We denote by $\qv{B}_t := ( \qv{B^i , B^j}_t )_{1 \le i,j \le d}$, 
$0 \le t \le T$, the quadratic variation of $B$.
For each $\eta \in ( M^1_G (\Omega) )^d$, we can define
\begin{align}
	\int_0^t ( d\qv{B}_s \, \eta_s )
	:= 
	\left(
	\sum_{j=1}^d \int_0^t \eta^j_s \, d\qv{B^i , B^j}_s
	\right)_{1 \le i \le d}
\end{align}
as an element of $\left( \calL^1_G (\Omega_t) \right)^d$.
Noting that $\eta^1 \eta^2 \in M^1_G (\Omega)$ for any 
$\eta^1 , \eta^2 \in M^2_G (\Omega)$,
we also set, for each $h \in ( M^2_G (\Omega) )^d$,
\begin{align}
	\int_0^t h_s \cdot (d \qv{B}_s \, h_s)
	:= \sum_{i,j=1}^d \int_0^t h^i_s \, h^j_s \, d \qv{B^i , B^j}_s .
\end{align}

\subsection{$\bm{G}$-martingales}

Now we introduce the notion of $G$-martingales.

\begin{defi}
	A process $X = \{ X_t ; 0 \le t \le T \}$ is called 
	a \bt{$\bm{G}$-martingale} if for each $0 \le s \le t \le T$,
	we have $X_t \in \calL^1_G (\Omega_t)$ and
	\begin{align}
		\E_s [X_t] 
		= X_s \qn{in } \calL^1_G (\Omega_s) .
	\end{align}
	We call $X$ a \bt{symmetric $\bm{G}$-martingale} 
	if both $X$ and $-X$ are $G$-martingales.
\end{defi}

For $h \in (M^2_G (\Omega))^d$, for example, 
It\^o's integral process $\int_0^{\cdot} h_s \cdot dB_s$ is a symmetric $G$-martingale.
$\int_0^{\cdot} h_s \cdot (d \qv{B}_s \, h_s) - \int_0^{\cdot} 2 G(h_sh^*_s) \, ds$ 
is a $G$-martingale, but in general,
not a symmetric $G$-martingale (see \cite{Peng;08b} Example 50).

\section{An upper expectation for $\bm{G}$-expectation}\label{Denis;10}

We introduce a representation of $G$-expectation 
as an upper expectation proved in Denis-Hu-Peng \cite{Denis;10}.

Let $W$ be a standard $d$-dimensional Brownian motion 
under a probability measure $P$ on $\Omega$, 
and let $\F^W$ be the filtration generated by $W$:
\begin{align}
	\calF^W_t 
	:= \sigma (W_u , 0 \le u \le t) \vee \calN ,
	\qd 
	\F^W 
	:= 
	\{ \calF^W_t ; t \ge 0 \} ,
\end{align}
where $\calN$ is the collection of all $P$-null subsets.
For a given bounded and closed set $\Theta \subset \R^{d \times d}$, let
\begin{align}
	\cA 
	:= 
	\{ 
	\n{all } \Theta \n{-valued } 
	\F^W \n{-progressively measurable processes on the interval } [0,T] 
	\} .
\end{align}
We identify two elements $\theta$, $\tp \in \cA$ if they are equivalent:
\begin{align}
	\theta_t ( \omega ) 
	= \tp_t ( \omega ) 
	\qd d t \times P \n{-a.e.\ } (t, \omega ) \in [0,T] \times \Omega .
\end{align}
The quotient set of $\cA$ by this equivalence relation 
is still denoted by the same symbol $\cA$.
For each $\theta \in \cA$, let $P_{\theta}$ 
be the law of the process $\{ \int_0^t \theta_s \, dW_s ; 0 \le t \le T \}$.
Now we define the \bt{capacity} $c : \calB (\Omega) \to [0,1]$ by
\begin{align}
	c (A) 
	:= 
	\sup_{\theta \in \cA} 
	P_{\theta} (A) 
	\qn{for } A \in \calB (\Omega) .
\end{align}
We introduce capacity-related terminology.
\begin{itemize}
	\item A property holds \bt{quasi-surely} (q.s.) 
		if it holds outside a set $A$ with $c(A)=0$.
	\item A mapping $X : \Omega \to \R$ 
		is said to be \bt{quasi-continuous} (q.c.) if for all $\ve >0$,
		there exists an open set $O$ with $c(O) < \ve$ 
		such that $X|_{O^c}$ is continuous.
	\item We say that $X : \Omega \to \R$ has a \bt{q.c.\ version} 
		if there exists a q.c.\ function $Y : \Omega \to \R$ with $X=Y$ q.s.
\end{itemize}
For $t \in [0,T]$, we denote by $L^0 (\Omega_t)$ 
the space of all $\calB (\Omega_t)$-measurable real-valued functions.
For $t =T$, we simply write $L^0(\Omega)$.
For each $X \in L^0 (\Omega)$ such that $E_{P_{\theta}} [X]$ 
exists for all $\theta \in \cA$, we set
\begin{align}
	\ol{\E} [X] 
	:= \sup_{\theta \in \cA} E_{P_{\theta}} [X] .
\end{align}

The following theorem plays a key role in the formulation 
and proof of Girsanov's formula.

\begin{thm}[\cite{Denis;10} Theorem~54]\label{chara}
	It holds that
	\begin{align}
		&\calL^1_G (\Omega_t) 
		= 
		\{ X \in L^0 (\Omega_t ) : \n{ X has a q.c.\ version, } 
		\lim_{n \to \infty} \ol{\E} [|X| \one_{ \{ |X|>n  \} }] =0 \} , \\
		&\E [X] 
		= \ol{\E} [X] 
		\qn{for all } X \in \calL^1_G (\Omega) .
	\end{align}
\end{thm}

\section{Main result}\label{Gir,section}

In this section, we  firstly characterize symmetric $G$-martingales.
We then state and prove the main result of this paper, Girsanov's formula for $G$-Brownian motion.
We also explore a condition that plays a similar role to Novikov's condition in the classical stochastic analysis.

\subsection{A characterization of symmetric $\bm{G}$-martingales}

We start with a lemma that characterizes conditional $G$-expectations.
For each $\theta \in \cA$ and $t \in [0,T]$, set
\begin{align}
	\calA (t,\theta ) 
	:= \{ \tp \in \cA : \tp = \theta ~\n{on } [0,t] \} ,
\end{align}
where the identity between $\tp$ and $\theta$ is to be understood as
\begin{align}
	\tp_s ( \omega ) 
	= \theta_s ( \omega ) 
	\qd d s \times P \n{-a.e.\ } (s,\omega ) \in [0,t] \times \Omega .
\end{align}

\begin{lem}\label{cG}
	For each $\theta \in \cA$, $X \in \calL^1_G (\Omega)$ 
	and $t \in [0, T]$, it holds that
	\begin{align}\label{esp}
		\E_t [X]
		= 
		\esp{ \tp \in \calA (t, \theta) }
		E_{P_{\tp}} [ X | \calF_t ]
		\qd P_{\theta}\n{-a.s.},
	\end{align}
	where $\{ \calF_t ; 0 \le t \le T \}$ is the natural filtration of $B$.
\end{lem}

As noted in the proof of Proposition 3.4 of Soner-Touzi-Zhang \cite{Soner;10a}, 
the validity of \eqref{esp} for $X \in \bLip ( \Omega )$ follows from \cite{Denis;10}; 
the assertion for $X \in \calL^1_G (\Omega)$ is then seen to hold 
by approximations as done in the proof of their proposition. 
Since there seems to be an inadequacy in its approximating argument 
and the family $\{ P_{\theta} : \theta \in \cA \}$ of probability measures
is strictly smaller than the one in their proposition, 
we give a proof of this lemma for the sake of self-containedness of the paper.

\begin{proof}[Proof of \lref{cG}]
	By Lemma 44 of \cite{Denis;10} and by the upper expectation
	representation for $G$-expectation (\tref{chara}), 
	we see that
	\begin{align}
		\left.
		\E [\vp (x, B^t_T)] 
		\right| 
		_{x = \zeta} 
		= 
		\esp{\theta \in \cA} 
		E_P [ \vp (\zeta , \tB{t}{T}) | \calF^W_t ] 
		\qd P \n{-a.s.}
	\end{align}
	for all $t \in [0,T]$, $m \in \N$, $\vp \in \bLip (\R^{m+d})$ 
	and $\zeta \in L^2 (\Omega , \calF^W_t , P ; \R^m)$.
	Here and below we write
	\begin{align}
		\tB{s}{t}
		= 
		\int _s^t \theta_u \, d W_u
		\qn{for } 0 \le s \le t \le T.
	\end{align}
	Then, repeating the same argument as in the proof of Theorem 45 of \cite{Denis;10}, 
	we see inductively that
	\begin{align}\label{dpp}
		\hspace*{-3mm}
		\left.
		\E [ \vp ( x , B^t_{s_1} , B^{s_1}_{s_2} \ddd B^{s_{k-1}}_{s_k} ) ] 
		\right|
		_{x = \zeta}
		= 
		\esp{\theta \in \cA} 
		E_P
		[ 
		\vp ( \zeta , \tB{t}{s_1} , \tB{s_1}{s_2} \ddd \tB{s_{k-1}}{s_k} ) 
		| \calF^W_t 
		]
	\end{align}
	$P$-a.s.\ for all 
	$t \in [0,T]$, $k, m \in \N$, $t \le s_1 < \dots < s_k \le T$, 
	$\vp \in \bLip (\R^m \times (\R^d)^k)$
	and $\zeta \in L^2 (\Omega , \calF^W_t , P ; \R^m)$.
	Now we fix $\theta \in \cA$ and $t \in [0,T)$ arbitrarily.
	We take $X= \vp (B_{t_1} \ddd B_{t_n}) \in \bLip (\Omega)$ 
	with a partition $0 = t_0 \le t_1 < \dots < t_n =T$,
	and let $i=0,1,\dots ,n-1$ be such that $t \in [t_i , t_{i+1})$.
	If we set
	\begin{align}
		\vp_1 (x_1 \ddd x_i , x)
		:= 
		\E [ \vp ( x_1 \ddd x_i , B^t_{t_{i+1}} + x \ddd B^t_{t_n} + x ) ]
	\end{align}
	for $(x_1 \ddd x_i , x) \in (\R^d)^{i+1}$, 
	then we have by \eqref{dpp}
	\begin{align}
		\vp_1 ( \tB{0}{t_1} \ddd \tB{0}{t_i} , \tB{0}{t} )
		= 
		\esp{\tp \in \calA (t,\theta)} 
		E_P [ \vp ( \tpB{0}{t_1} \ddd \tpB{0}{t_n} ) | \calF^W_t ]
		\qd P \n{-a.s.}
	\end{align}
	Let $U \in \calF_t$ be arbitrary and 
	set $V= \{ \tB{0}{\cdot} \in U \} \in \calF^W_t$. 
	Then
	\begin{align}
		E_{P_{\theta}} 
		[ \one_U \vp_1 (B_{t_1} \ddd B_{t_i} , B_t) ]
		&= 
		E_P 
		[ \one_V \vp_1 ( \tB{0}{t_1} \ddd \tB{0}{t_i} , \tB{0}{t} ) ]\\
		&= 
		E_P 
		[ \one_V 
		\esp{\tp \in \calA (t,\theta)} 
		E_P [ \vp ( \tpB{0}{t_1} \ddd \tpB{0}{t_n} ) | \calF^W_t ]
		] \\
		&= 
		\sup_{\tp \in \calA (t,\theta)} 
		E_P [ \one_V \vp ( \tpB{0}{t_1} \ddd \tpB{0}{t_n} ) ]\\
		&= 
		\sup_{\tp \in \calA (t,\theta)} 
		E_{P_{\tp}} 
		[ \one_U \vp ( B_{t_1} \ddd B_{t_n} ) ] ,
	\end{align}
	where we used Yan's commutation theorem 
	(see, e.g., \cite{Peng;04} Theorem~a3) for the third line.
	Using Yan's commutation theorem again, 
	and noting $P_{\theta} = P_{\tp}$ on $\calF_t$ for $\tp \in \calA (t, \theta)$,
	we see that this is further rewritten as
	\begin{align}
		E_{P_{\theta}} 
		[ \one_U 
		\esp{\tp \in \calA (t,\theta)} 
		E_{P_{\tp}} 
		[\vp (B_{t_1} \ddd B_{t_n}) | \calF_t ] 
		] .
	\end{align}
	As $\vp_1 (B_{t_1} \ddd B_{t_i} , B_t) = \E_t [\vp (B_{t_1} \ddd B_{t_n})]$ 
	by definition, it follows that
	\begin{align}
		\E_t [\vp (B_{t_1} \ddd B_{t_n})]
		= 
		\esp{\tp \in \calA (t,\theta)} 
		E_{P_{\tp}} 
		[\vp (B_{t_1} \ddd B_{t_n}) | \calF_t] 
		\qd P_{\theta} \n{-a.s.}
	\end{align}
	Therefore \eqref{esp} is proved for $X \in \bLip (\Omega)$.

	Now for $X \in \calL^1_G (\Omega)$, 
	we take a sequence $\{ X_n \}_{n=1}^{\infty} \subset \bLip (\Omega)$ such that
	\begin{align}
		\E [|X-X_n|] \to 0 
		\qn{as } n \to \infty .
	\end{align}
	For each $\theta \in \cA$,
	\begin{align}
		&E_{P_{\theta}} 
		[ | 
		\E_t [X] 
		- \esp{\tp \in \calA (t,\theta)} 
		E_{P_{\tp}} [X | \calF_t ]
		| ] \\
		&\le E_{P_{\theta}} 
		[|\E_t [X] - \E_t [X_n]|]
		+ 
		E_{P_{\theta}} 
		[|
		\esp{\tp \in \calA (t,\theta)} 
		E_{P_{\tp}} [X | \calF_t ] 
		- 
		\esp{\tp \in \calA (t, \theta)} 
		E_{P_{\tp}}[X_n | \calF_t ]
		|]\\
		&=: 
		I_n + I\!I _n.
	\end{align}
	It is easily seen that $I_n \le \E [|X-X_n|]$.
	Also for $I\!I_n$, we have
	\begin{align}
		I\!I _n
		&\le 
		E_{P_{\theta}} 
		[ 
		\esp{\tp \in \calA (t,\theta)} 
		E_{P_{\tp}} [ | X - X_n | | \calF_t ] 
		]\\
		&=
		\sup_{\tp \in \calA (t,\theta)} 
		E_{P_{\tp}} [ | X - X_n | ]\\
		&\le 
		\E [ | X - X_n | ] ,
	\end{align}
	where the equality follows from Yan's commutation theorem 
	and the identity $P_{\theta} = P_{\tp}$ on $\calF_t$ for $\tp \in \calA (t,\theta)$.
	Therefore both $I_n$ and $I\!I_n$ converge to $0$ as $n \to \infty$, 
	which yields \eqref{esp} for $X \in \calL^1_G (\Omega)$.
\end{proof}

As a consequence of \lref{cG}, we have the following characterization of symmetric $G$-martingales.

\begin{prop}\label{symG}
	$X = \{ X_t ; 0 \le t \le T \}$ is a symmetric $G$-martingale 
	on $(\Omega , \calL^1_G (\Omega) , \E)$
	if and only if $X_t \in \calL^1_G (\Omega_t)$ 
	for all $t \in [0,T]$ and $X$ is a $P_{\theta}$-martingale 
	for each $\theta \in \cA$.
\end{prop}

\begin{proof}
	We start with the if part.
	The condition that $X_t \in \calL^1_G (\Omega_t) ,~ t \in [0,T]$, 
	means that $X$ is a process on $(\Omega , \calL^1_G (\Omega) , \E)$.
	If $X$ is also a $P_{\theta}$-martingale for each $\theta \in \cA$, 
	we have, for $0 \le s \le t \le T$,
	\begin{align}
		X_s 
		= 
		\esp{\tp \in \calA (s, \theta)} 
		E_{P_{\tp}} [ X_t | \calF_s ] 
		\qd P_{\theta} \n{-a.s.}
	\end{align}
	By \lref{cG}, it follows that $X_s = \E_s [X_t] ~ P_{\theta}$-a.s.\ and that
	\begin{align}
		\E [ | \E_s [X_t] - X_s | ] = 0 .
	\end{align}
	Similarly, we have $\E_s [-X_t] = -X_s$ in $\calL^1_G (\Omega_s)$ 
	and hence $X$ is a symmetric $G$-martingale.

	Conversely, if $X$ is a symmetric $G$-martingale, 
	then $X_t \in \calL^1_G (\Omega_t)$ for all $t \in [0,T]$.
	Since $X$ is a $G$-martingale,
	\begin{align}
		0 
		= 
		\E [ | \E_s [X_t] - X_s | ] 
		= 
		\sup_{\theta \in \cA} 
		E_{P_{\theta}} [ | \E_s [X_t] - X_s | ] .
	\end{align}
	Therefore, for every $\theta \in \cA$, we have by \lref{cG},
	\begin{align}
		X_s 
		= 
		\E_s [X_t] 
		= 
		\esp{\tp \in \calA (t, \theta)} 
		E_{P_{\tp}} [X_t | \calF_s ] 
		\ge 
		E_{P_{\theta}} [X_t | \calF_s ] 
		\qd P_{\theta} \n{-a.s.}
	\end{align}
	Similarly, we deduce $X_s \le E_{P_{\theta}} [X_t | \calF_s ]$ 
	$P_{\theta}$-a.s.\ from that $-X$ is a $G$-martingale.
	Hence, $X$ is a $P_{\theta}$-martingale for each $\theta \in \cA$.
\end{proof}

\subsection{Girsanov's formula for $\bm{G}$-Brownian motion}\label{junbi}

Let $h \in (M^2_G (\Omega))^d$.
We define, for $0 \le t \le T$,
\begin{align}\label{D}
	&D_t 
	:= 
	\exp 
	\left( 
	\int_0^t h_s \cdot dB_s 
	- \frac{1}{2} \int_0^t h_s \cdot (d \qv{B}_s \, h_s) 
	\right) , \\
	&\hat{B}_t 
	:= 
	B_t - \int_0^t (d \qv{B}_s \, h_s) ,
\end{align}
and we set 
\begin{align}
	\hat{C}_{b,Lip} (\Omega) 
	:= 
	\{ 
	\vp ( \hat{B}_{t_1} \ddd \hat{B}_{t_n}) : 
	n \in \N ,~ t_1 \ddd t_n \in [0,T] ,~ \vp \in \bLip ( ( \R^d )^n) 
	\} .
\end{align}
As $\hat{B}_t \in (\calL^1_G ( \Omega_t ))^d$ for each $t \in [0,T]$,
we may deduce from \tref{chara} that 
$\hat{C}_{b,Lip} ( \Omega )$ is a subspace of $\calL^1_G ( \Omega )$.

Girsanov's formula for $G$-Brownian motion is stated as follows.

\begin{thm}\label{GGir,thm}
	Assume that there exists $\sigma_0 >0$ such that 
	\begin{align}\label{unif,nondege,eq}
		\gamma \gamma^* 
		\ge 
		\sigma_0 I_d 
		\qn{for all } \gamma \in \Theta , 
	\end{align}
	and that $D$ is a symmetric $G$-martingale on $(\Omega , \calL^1_G (\Omega) , \E)$.
	Define a sublinear expectation $\hat{\E}$ by
	\begin{align}\label{hatE}
		\hat{\E} [X] 
		:= 
		\E [X D_T] 
		\qn{for } X \in \hat{C}_{b,Lip} ( \Omega ) .
	\end{align}
	Let $\hat{\calH}$ be the completion of $\hat{C}_{b,Lip} (\Omega)$ 
	under the norm $\hat{\E} [| \cdot |]$,
	and extend $\hat{\E}$ to a unique sublinear expectation on $\hat{\calH}$.
	Then the process $\{ \hat{B}_t ; 0 \le t \le T \}$ is a $G$-Brownian motion 
	on the sublinear expectation space $( \Omega , \hat{\calH} , \hat{\E} )$.
\end{thm}

We remark that the uniform nondegeneracy of $\Theta$ is also assumed in \cite{Xu;10}.

We prove \tref{GGir,thm} in the next subsection.
Before we proceed to the proof, there are several things we must verify.
The first thing is the well-definedness of the right-hand side of \eqref{hatE},
which is immediate from \tref{chara} since $D_T$ is in $\calL^1_G ( \Omega )$ 
by assumption and $X$ is a bounded element of $\calL^1_G ( \Omega )$. 
The second is that the functional $\hat{\E}$ defined by \eqref{hatE} is indeed a sublinear expectation.
As the assumption on $D$ also yields $\E [ D_T ] = - \E [ -D_T ] =1$, 
this functional possesses the property (ii) in \dref{slinear,def}.
The other three properties follow readily from the definition.
The last thing to be verified prior to the proof of \tref{GGir,thm}
is that $\{ \hat{B}_t ; 0 \le t \le T \}$ is a stochastic process on 
$(\Omega , \hat{\calH} , \hat{\E})$,
which we will check in the next lemma.
For a fixed $\theta \in \cA$, set
\begin{align}\label{Q}
	Q_{\theta} (A) 
	:= 
	E_{P_{\theta}} [\one_A D_T] 
	\qn{for }A \in \calB (\Omega).
\end{align}
Note that, by \tref{chara}, we have
\begin{align}\label{upperQ}
	\hat{\E} [X] 
	= 
	\sup_{\theta \in \cA} E_{Q_{\theta}} [X]
\end{align}
for all $X \in \hat{C}_{b,Lip} (\Omega)$.

\begin{lem}\label{process}
	For all $t \in [0,T]$, we have $\hat{B}_t \in \hat{\calH}^d$. 
	Therefore $\hat{B}$ is a stochastic process on $(\Omega , \hat{\calH} , \hat{\E})$.
\end{lem}

\begin{proof}
	Fix $i = 1 \ddd d$ and take an arbitrary $\theta \in \cA$.
	By definition, the $i$-th coordinate $B^i$ 
	of the canonical process $B$ is a $P_{\theta}$-martingale.
	Note that the process $D$ is also a $P_{\theta}$-martingale by \pref{symG}
	and satisfies the following relation with $B^i$ and $\hat{B}^i$:
	\begin{align}
		\hat{B}^i_t 
		= B^i_t - \int_0^t \frac{ d \qv{D,B^i}_s }{D_s} .
	\end{align}
	Therefore, by Girsanov's formula, $\hat{B}^i$ 
	is a local martingale under $Q_{\theta}$ and
	\begin{align}
		\qv{\hat{B}^i}_t 
		= \qv{B^i}_t 
		\qn{for all } t \in [0,T] ,~ 
		Q_{\theta} \n{-a.s.\ and } P_{\theta} \n{-a.s.}
	\end{align}
	By definition, $\qv{B^i}_T$ under $P_{\theta}$ is identical in law 
	with $\int_0^T (\theta_s \theta^*_s )^{ii} \, ds$,
	where $(\theta_s \theta^*_s )^{ii}$ is 
	the $(i,i)$-entry of the matrix $\theta_s \theta^*_s$.
	We thus deduce that, by the boundedness of $\Theta$, 
	there exists a constant $C>0$ depending only on $\Theta$ such that
	\begin{align}\label{qv}
		\qv{\hat{B^i}}_T 
		\le 
		CT 
		\qd Q_{\theta} \n{-a.s.}
	\end{align}
	Moreover, by the time-change formula due to Dambis-Dubins-Schwarz 
	(see, e.g., \cite{Karatzas;91} Theorem 3.4.6),
	there exists a standard Brownian motion $\beta$ 
	under $Q_{\theta}$ such that
	\begin{align}
		\hat{B}^i_t 
		= 
		\beta_{\qv{\hat{B}^i}_t} 
		\qn{for all } t \in [0,T] ,~ Q_{\theta} \n{-a.s.}
	\end{align}
	Combining these, we have, for some $p>1$ (actually, for all $p>1$),
	\begin{align}\label{unif,hatB}
		\sup_{\theta \in \cA} 
		E_{Q_{\theta}} [|\hat{B}^i_t|^p]
		\le 
		\sup_{\theta \in \cA} 
		E_{Q_{\theta}} [ \max_{0 \le t \le CT} | \beta_t |^p]
		= 
		E_P [\max_{0 \le t \le CT} |W_t|^p]
		< 
		\infty ,
	\end{align}
	where $W$ is a one-dimensional Brownian motion under a probability measure $P$.

	Now define the sequence 
	$\{ \vp_n ( \hat{B}_t^i ) \}_{n=1}^{\infty} \subset \hat{C}_{b,Lip} (\Omega)$ through
	\begin{align}
		\vp_n (x) 
		:= 
		(x \wedge n) \vee (-n) 
		\qn{for } x \in \R .
	\end{align}
	This approximates $\hat{B}^i_t$ under the norm $\hat{\E} [| \cdot |]$.
	Indeed, by \eqref{unif,hatB}
	\begin{align}\label{dense}
		\sup_{\theta \in \cA} 
		E_{Q_{\theta}} [|\hat{B}^i_t - \vp_n (\hat{B}^i_t)|]
		\le 
		\sup_{\theta \in \cA} 
		E_{Q_{\theta}} [|\hat{B}^i_t| \one_{\{ |\hat{B}^i_t| >n \}}]
		\to 0 \qd (n \to \infty ) .
	\end{align}
	By noting that, from \eqref{upperQ},
	$\hat{\calH}$ can be seen as the completion of 
	$\hat{C}_{b,Lip} (\Omega)$ under the norm 
	$\sup_{\theta \in \cA} E_{Q_{\theta}} [| \cdot |]$,
	\eqref{dense} shows $\hat{B}^i_t \in \hat{\calH}$.
\end{proof}

In the proof of \tref{GGir,thm}, 
it will also be required that $\hat{B}$ is a true martingale under $Q_{\theta}$,
which follows immediately from \eqref{qv} and Corollary IV.1.25 of \cite{Revuz;99}.
We state it in the lemma.

\begin{lem}\label{hatB,mar}
	For each $\theta \in \cA$, 
	the process $\{ \hat{B}_t ; 0 \le t \le T \}$ is a $Q_{\theta}$-martingale.
\end{lem}

\subsection{Proof of \tref{GGir,thm}}

A probability measure $P$ on $(\Omega , \calB (\Omega) )$ is called a \bt{martingale measure}
if the canonical process $B$ is a martingale with respect to $\F^B$ under $P$, 
where $\F^B$ is the filtration generated by $B$:
\begin{align}
	\calF^B_t 
	:= 
	\sigma( B_u , 0 \le u \le t ) \vee \calN 
	, \qd 
	\F^B 
	:= 
	\{ \calF^B_t ; 0 \le t \le T \} ,
\end{align}
where $\calN$ is the collection of all $P$-null subsets. 
Let $\scrP$ be the family of all martingale measures $P$ satisfying
\begin{align}
	\frac{d \qv{B}^P_t}{dt} 
	\in 
	\{ \gamma \gamma^* : \gamma \in \Theta \} 
	, \qn{a.e.\ } t \in [0,T] ,~ P \n{-a.s.,}
\end{align}
where $\qv{B}^P$ is the quadratic variation process of $B$ under $P$. 
First we prove

\begin{lem}\label{sup,mar,lem}
	For all $X \in \bLip (\Omega)$,
	\begin{align}
		\E [X] = \sup_{P \in \scrP} E_P [X] .
	\end{align}
\end{lem}

In the case that the set $\Theta \subset \R^{d \times d}$ has a form
$\{ \gamma \in \R^{d \times d}  : \sigma_0 I_d \le \gamma \gamma^* \le \sigma_1 I_d \}$
for some constants $0 < \sigma_0 \le \sigma_1$, 
this lemma follows readily from Proposition~3.4 in \cite{Soner;10a}.
Notice that the proof below does not use any structures of $\Theta$ 
other than uniform nondegeneracy \eqref{unif,nondege,eq}. 

\begin{proof}[Proof of \lref{sup,mar,lem}]
	Since $\{ P_{\theta} : \theta \in \cA \} \subset \scrP$, 
	it is clear that $\E [X] = \ol{\E} [X] \le \sup_{P \in \scrP} E_P [X]$.
	We check the reverse inequality
	\begin{align}\label{rev,ineq}
		\E [X] \ge \sup_{P \in \scrP} E_P [X] .
	\end{align}
	For each $n \in \N$, we set the statement $\mathfrak{p} (n)$ as follows:
	\begin{align}
		\mathfrak{p} (n) : 
		&\n{ For all } 0 \le t_1 < \dots < t_n \le T 
		, \n{ and } \vp \in \bLip ((\R^d)^n) ,\\
		&\sup_{P \in \scrP} 
		E_P [\vp (B_{t_1} \ddd B_{t_n})] 
		\le \E[\vp (B_{t_1} \ddd B_{t_n})] 
		\qn{holds.}
	\end{align}
	We show \eqref{rev,ineq} by induction with respect to $n$.

	(i) First we let $n=1$, and $v$ be the solution of the following $G$-heat equation:
	\begin{align}
		\left\{
		\begin{aligned}
			&- \frac{\partial v}{\partial t} - G(D^2 v) =0 
			\qn{in } (0,t_1) \times \R^d , \\
			&v |_{t=t_1} = \vp \qn{in } \R^d .
		\end{aligned}
		\right.
	\end{align}
	Note that $v \in C^{1,2} ((0,t_1) \times \R^d)$ 
	by assumption \eqref{unif,nondege,eq} (see \rref{hitaika}).
	For all $P \in \scrP$, it follows from It\^o's formula that $P$-a.s.
	\begin{align}
		\vp (B_{t_1}) 
		&= 
		v (t_1 , B_{t_1}) \\
		&= 
		v(0,0) + \int_0^{t_1} (D v)(t , B_t) \cdot dB_t \\
		&\qd 
		+ \int_0^{t_1} 
		\left( 
		-G((D^2 v)(t , B_t)) \, d t 
		+ \frac{1}{2} \tr \big[ (D^2 v)(t , B_t) \, d \qv{B}^P_t \big] 
		\right) \\
		&\le 
		v(0,0) + \int_0^{t_1} (D v)(t , B_t) \cdot dB_t .
	\end{align}
	Taking the expectation under $P$, 
	we have $E_P [\vp (B_{t_1})] \le v(0,0) = \E [\vp (B_{t_1})]$. 
	Hence
	\begin{align}
		\sup_{P \in \scrP} 
		E_P [\vp (B_{t_1})] 
		\le 
		\E [\vp (B_{t_1})] .
	\end{align}

	(ii) We now assume that $\mathfrak{p} (n)$ is true for some $n \in \N$.
	Take $0 \le t_1 < \cdots < t_n < t_{n+1} \le T$ and 
	$\vp \in C_{b,Lip} ((\R^d)^{n+1})$ to be arbitrary.
	By the definition of conditional $G$-expectations, it holds that
	\begin{align}
		\E_{t_n} [\vp (B_{t_1} \ddd B_{t_n} , B_{t_{n+1}}) ] 
		= v (t_n , B_{t_n} ; B_{t_1} \ddd B_{t_n}),
	\end{align}
	where $v(t,x ; x_1 \ddd x_n) \in C^{1,2} ((t_n , t_{n+1}) \times \R^d)$ 
	is the solution of the following $G$-heat equation:
	\begin{align}
		\left\{
		\begin{aligned}
			&- \frac{\partial v}{\partial t} - G(D_x^2 v) = 0 
			\qn{in } (t_n , t_{n+1}) \times \R^d , \\
			&v(t_{n+1} , x ; x_1 \ddd x_n) = \vp (x_1 \ddd x_n ,x) ,
			\qd x \in \R^d .
		\end{aligned}
		\right.
	\end{align}
	Under each $P \in \scrP$, we apply It\^o's formula (see \rref{ito}) 
	to $v(t_{n+1} , B_{t_{n+1}} ; B_{t_1} \ddd B_{t_n})$ to obtain $P$-a.s.
	\begin{align}
			&\vp (B_{t_1} \ddd B_{t_n} , B_{t_{n+1}})
			= 
			v( t_{n+1} , B_{t_{n+1}} ; B_{t_1} \ddd B_{t_n} ) \\
			&= 
			v ( t_n , B_{t_n} ; B_{t_1} \ddd B_{t_n} )
			+ \int_{t_n}^{t_{n+1}} 
			(D_x v) ( t , B_t ; B_{t_1} \ddd B_{t_n} ) \cdot dB_t \\
			&\qd 
			+ \int_{t_n}^{t_{n+1}} 
			\left( - G( (D_x^2 v)( t , B_t ; B_{t_1} \ddd B_{t_n} ) ) \, d t
			+\frac{1}{2} 
			\tr \big[ 
			(D_x^2 v) ( t , B_t ; B_{t_1} \ddd B_{t_n} ) \, d \qv{B}^P_t 
			\big] 
			\right) \\
			&\le 
			v ( t_n , B_{t_n} ; B_{t_1} \ddd B_{t_n} )
			+ \int_{t_n}^{t_{n+1}} 
			(D_x v) ( t , B_t ; B_{t_1} \ddd B_{t_n} ) \cdot dB_t \, . 
			\label{ito,eq}
	\end{align}
	Taking the expectation under $P$, we have
	\begin{align}
		E_P [ \vp ( B_{t_1} \ddd B_{t_n} , B_{t_{n+1}} ) ]
		\le 
		E_P [ v( t_n , B_{t_n} ; B_{t_1} \ddd B_{t_n} ) ] .
	\end{align}
	Therefore
	\begin{align}
		\sup_{P \in \scrP} 
		E_P [ \vp ( B_{t_1} \ddd B_{t_n} , B_{t_{n+1}} ) ]
		\le 
		\sup_{P \in \scrP} 
		E_P [ v ( t_n , B_{t_n} ; B_{t_1} \ddd B_{t_n} ) ] .
	\end{align}
	Notice that the function 
	$(x_1 \ddd x_n) \mapsto v(t_n , x_{n} ; x_1 \ddd x_n)$ 
	belongs to $C_{b,Lip} ((\R^d)^{n})$.
	By the assumption that $\mathfrak{p} (n)$ is true, we have
	\begin{align}
		\sup_{P \in \scrP} 
		E_P [ v ( t_n , B_{t_n} ; B_{t_1} \ddd B_{t_n} ) ]
		&\le \E[v(t_n , B_{t_n} ; B_{t_1} \ddd B_{t_n})] \\
		&= \E[\vp (B_{t_1} \ddd B_{t_n} , B_{t_{n+1}})] .
	\end{align}
	So $\mathfrak{p} (n+1)$ is also true, and hence we complete the induction argument.
\end{proof}

\begin{rem}\label{ito}
	The second equality in \eqref{ito,eq} may be seen in the following manner:
	for each $i = 1\ddd n$, define the process $M^i$ on $[ t_n , t_{n+1} ]$ by
	\begin{align}
		M^i_t := B_{t_i} , 
		\qd t_n \le t \le t_{n+1},
	\end{align}
	and set $M_t := (t, B_t , M^1_t \ddd M^n_t)$.
	Clearly $\{ M_t ; t_n \le t \le t_{n+1} \}$ is an $\F^B$-semimartingale.
	We may write $v(M_t)$ for $v(t,B_t ; B_{t_1} \ddd B_{t_n})$,
	to which It\^o's formula applies to yield the desired equality.
\end{rem}

Now we are in a position to prove \tref{GGir,thm}.

\begin{proof}[Proof of \tref{GGir,thm}]
	It is sufficient to show that for all $k \in \N,~t_1 \ddd t_k \in [0,T]$, 
	and $\vp \in C_{b,Lip} ((\R^d)^k)$,
	\begin{flalign}
		\hat{\E} [\vp (\hat{B}_{t_1} \ddd \hat{B}_{t_k})] 
		= \E [\vp (B_{t_1} \ddd B_{t_k})] .
	\end{flalign}
	Indeed, it is obvious that $\hat{B}$ satisfies \dref{GBm,def} (i).
	If we obtain the above equation, 
	it then follows from the right-hand side that $\hat{B}$ 
	satisfies Definition~\ref{GBm,def} (ii), (iii) under $\hat{\E}$.
	Note that, since $\hat{\calH}$ is the completion of $\hat{C}_{b,Lip} (\Omega)$,
	identical distributedness and independence on $(\Omega , \hat{\calH} , \hat{\E})$ can be,
	as those on $( \Omega , \calL^1_G ( \Omega ), \E)$ are, 
	checked through test functions of the class consisting of bounded,
	Lipschitz cylinder functionals (see \dref{dis} and the comment given just after it).

	For simplicity, we write $\vp (B)$ and $\vp (\hat{B})$ for 
	$\vp (B_{t_1} \ddd B_{t_k})$ and 
	$\vp (\hat{B}_{t_1} \ddd \hat{B}_{t_k})$, respectively.

	(i) First we show that $\E [\vp (B)] \le \hat{\E} [\vp (\hat{B})]$.

	It is enough to show the following:
	\begin{flalign}\label{s,apploxi,eq}
		\n{for all } \Theta \n{-valued {\em simple} process } 
		\theta \n{ on } [0,T] ,~ 
		E_{P_{\theta}} [\vp (B)] 
		\le 
		\hat{\E} [\vp (\hat{B})] .
	\end{flalign}
	To see this, we fix $\theta \in \cA$.
	Then, for all $\ve >0$, there exists a $\Theta$-valued simple process 
	$\theta^{\ve}$ on $[0,T]$ such that
	\begin{flalign}
		E_P [\int_0^T \| \theta^{\ve}_s - \theta_s \|^2 \, d s] 
		< \ve^2
	\end{flalign}
	(see, e.g., \cite{Karatzas;91} Problem~3.2.5). 
	Therefore, if \eqref{s,apploxi,eq} holds, we have
	\begin{flalign}
		&E_{P_{\theta}} 
		[\vp (B)] 
		\equiv 
		E_{P_{\theta}} 
		[\vp (B_{t_1} \ddd B_{t_k})] \\
		&\le 
		E_{P_{\theta^{\ve}}}
		[\vp (B_{t_1} \ddd B_{t_k})]
		+
		C_{\vp} 
		E_P 
		\Big[ \, \Big( 
		\, \sum_{i=1}^k \sum_{j=1}^d \, 
		\Big| 
		\sum_{l=1}^d 
		\int_0^{t_i} ( \theta_s - \theta^{\ve}_s)^{j l} \, d W^l_s \, 
		\Big|^2 \, 
		\Big)^{1/2} \Big]\\
		&\le 
		\hat{\E} 
		[\vp (\hat{B}_{t_1} \ddd \hat{B}_{t_k})]
		+ 
		C_{\vp} 
		E_P 
		\Big[ \, 
		\sum_{i=1}^k \sum_{j=1}^d \, 
		\Big| 
		\sum_{l=1}^d 
		\int_0^{t_i} ( \theta_s - \theta^{\ve}_s )^{j l} \, d W^l_s \, 
		\Big|^2 
		\, \Big]^{1/2}\\
		&\le 
		\hat{\E} 
		[\vp (\hat{B})] 
		+
		C_{\vp} 
		\sqrt{k} 
		E_P 
		\Big[
		\int_0^T \| \theta_s - \theta_s^{\ve} \|^2 \, d s 
		\Big] ^{1/2} \\
		&\le 
		\hat{\E} 
		[\vp (\hat{B})] 
		+ 
		C_{\vp} 
		\sqrt{k} \ve ,
	\end{flalign}
	where $C_{\vp}$ is a Lipschitz constant of $\vp$ and 
	$(\theta_s - \theta_s^{\ve})^{jl}$ is the $(j,l)$-entry of $\theta_s - \theta_s^{\ve}$.
	Since $\ve >0$ is arbitrary, we get
	\begin{align}
		E_{P_{\theta}} [\vp (B)]
		\le 
		\hat{\E} [\vp (\hat{B})].
	\end{align}
	Now we show \eqref{s,apploxi,eq}. Let $\theta$ be given in the form
	\begin{flalign}\label{theta}
		\theta_t 
		= 
		\eta_0 \one_{[t_0,t_1]}(t) 
		+ 
		\eta_1(W) \one_{(t_1,t_2]}(t) 
		+ \dots + 
		\eta_{n-1}(W) \one_{(t_{n-1},t_n]}(t)
	\end{flalign}
	for $0 \le t \le T$, where $0=t_0<t_1< \cdots <t_n=T$ is a partition of $[0,T]$, 
	$\eta_0 \in \Theta$,
	and $\eta_i (\omega) \equiv \eta_i (\omega_t ,~ t \le t_i)$, 
	$\omega \in \Omega$,
	is a $\Theta$-valued measurable functional on $\Omega$ for $i = 1, \dots ,n-1$.
	We now define the sequence of random variables $\{ \tilde{\eta}_i \}_{i=1}^{n-1}$
	and the simple process $\tilde{\theta} = \{ \tilde{\theta}_t ; 0 \le t \le T \}$ as follows:
	\begin{flalign}
		\left\{
		\begin{aligned}
			\tilde{\eta}_0 
			&:= 
			\eta_0, \\
			\tilde{\eta}_1 
			&:= 
			\eta_1 (W_t - 
			\int_0^t  \tilde{\theta}_s^* \, h_s^{(\tilde{\theta})} \, d s ,~ 
			t \le t_1), \\
			&\vdots \\
			\tilde{\eta}_{n-1} 
			&:= 
			\eta_{n-1} (W_t - 
			\int_0^t \tilde{\theta}_s^* \, h_s^{(\tilde{\theta})} \, d s ,~ 
			t \le t_{n-1} ),
		\end{aligned}
		\qd
		\begin{aligned}
			\tilde{\theta}_t 
			&:= 
			\tilde{\eta}_0 ,~ 
			t_0 \le t \le t_1 ,\\
			\tilde{\theta}_t 
			&:= 
			\tilde{\eta}_1 ,~
			t_1 < t \le t_2 ,\\
			&\vdots \\
			\tilde{\theta}_t 
			&:= 
			\tilde{\eta}_{n-1} ,~
			t_{n-1} < t \le t_n ,
		\end{aligned}
		\right.
	\end{flalign}
	where $h_s^{(\tilde{\theta})} := h_s (\int_0^{\cdot} \tilde{\theta}_u \, dW_u)$.
	As the right-hand side of \eqref{theta} is given as a functional of $W$, 
	we denote it by $\theta_t (W)$ with a slight abuse of notation.
	Then, from the above construction of $\tilde{\theta}$, for all $0 \le t \le T$,
	\begin{flalign}
		\tilde{\theta}_t
		= \theta_t ( W - 
		\int_0^{\cdot} \tilde{\theta}_s^* \, h_s^{(\tilde{\theta})} \, d s ).
	\end{flalign}
	Set
	\begin{flalign}
		&\widetilde{W}_t 
		:= 
		W_t- \int_0^t \tilde{\theta}_s^* \, h_s^{(\tilde{\theta})} \, d s, 
		\qd 0 \le t \le T ,\\
		&D^{( \tilde{\theta} )}_T 
		:= 
		\exp 
		\left(
		\int_0^T 
		\tilde{\theta}_t^* \, h_t^{(\tilde{\theta})} \cdot d W_t
		-\frac{1}{2} \int_0^T 
		h_t^{(\tilde{\theta})} \cdot
		(\tilde{\theta}_t \tilde{\theta}_t^* \, h_t^{(\tilde{\theta})}) \,
		d t
		\right) ,\\
		&\widetilde{P}(A) 
		:= 
		E_P[ \one_A D^{( \tilde{\theta} )}_T] , 
		\qd A \in \calF^W_T .
	\end{flalign}
	Since, by Girsanov's formula, $\widetilde{W}$ 
	is a Brownian motion under $\widetilde{P}$, we have
	\begin{flalign}
		E_{P_{\theta}} 
		[\vp (B)]
		&= 
		E_{\widetilde{P}} 
		[ \vp (\int_0^{\cdot} \theta_s(\widetilde{W}) \, d \widetilde{W}_s) ] \\
		&=
		E_P 
		[ 
		\vp ( \int_0^{\cdot} \tilde{\theta}_s \, d W_s
		- \int_0^{\cdot} 
		\tilde{\theta}_s \tilde{\theta}_s^* \, h_s^{(\tilde{\theta})} \, d s ) 
		\, D^{( \tilde{\theta} )}_T 
		] \\
		&= 
		E_{P_{\tilde{\theta}}} 
		[ 
		\vp ( B - \int_0^{\cdot} (d \qv{B}_s \, h_s) ) 
		\, D_T
		] \\
		&\le 
		\E 
		[ \vp ( B - \int_0^{\cdot} (d \qv{B}_s \, h_s) ) \, D_T ] 
		= 
		\hat{\E} [\vp (\hat{B})] ,
	\end{flalign}
	which shows \eqref{s,apploxi,eq}.

	(ii) Next we show that $\hat{\E} [\vp (\hat{B})] \le \E [\vp (B)]$.

	For each $\theta \in \cA$, let $Q_{\theta}$ be the measure defined by \eqref{Q}.
	By \lref{hatB,mar}, $\hat{B}$ is a $Q_{\theta}$-martingale.
	Girsanov's formula also implies that
	\begin{flalign}
		\qv{ \hat{B} } 
		= \qv{B} 
		,\qd P_{\theta} \n{-a.s.\ and }Q_{\theta} \n{-a.s.}
	\end{flalign}
	Hence $Q_{\theta} \circ \hat{B}^{-1} \in \scrP$,
	where $Q_{\theta} \circ \hat{B}^{-1} (A) := Q_{\theta} (\hat{B} \in A)$ 
	for each $A \in \calB (\Omega)$.
	Then, using \lref{sup,mar,lem}, we have
	\begin{flalign}
		E_{P_{\theta}} 
		[\vp (\hat{B}) D_T]
		= 
		E_{Q_{\theta} \circ \hat{B}^{-1}} 
		[\vp (B)]
		\le 
		\sup_{P \in \scrP} 
		E_P [\vp (B)]
		= 
		\E [\vp (B)] .
	\end{flalign}
	Therefore we get
	\begin{align}
		\hat{\E} [\vp (\hat{B}) ] 
		= 
		\sup_{\theta \in \cA} 
		E_{P_{\theta}} 
		[\vp (\hat{B}) D_T] 
		\le 
		\E [\vp (B)] ,
	\end{align}
	and complete the proof.
\end{proof}

\begin{rem}\label{d:rem}
	In \cite{Xu;10a}, a L\'evy-type characterization 
	of one-dimensional $G$-Brownian motion is given, 
	and by using that characterization, 
	Xu-Shang-Zhang \cite{Xu;10} obtains Girsanov's formula 
	for one-dimensional $G$-Brownian motion. 
	On the other hand, as far as we know, 
	such a characterization is not available in the case of multidimension. 
	We remark that unlike the classical Brownian motion, 
	components of multidimensional $G$-Brownian motion 
	are correlated due to variance uncertainty. 
	The advantage of our method 
	is to appeal directly to the definition of $G$-Brownian motion (\dref{GBm,def}),
	which enables us to deal with the multidimensional case. 
\end{rem}

We conclude this subsection with a remark on the construction of $\hat{\E}$.

\begin{rem}
	The equation \eqref{hatE} holds on $\hat{\calH}$, namely
	\begin{align}
		X D_T \in \calL^1_G (\Omega) 
		\qn{for all } X \in \hat{\calH} .
	\end{align}
	To see this, it is sufficient to check the completeness of 
	$\scrL := \{ X \in \calL^1_G (\Omega) : X D_T \in \calL^1_G (\Omega) \}$
	with respect to the norm $\hat{\E} [| \cdot |]$.
	Let $\{ X_n \}_{n=1}^{\infty} \subset \mathscr{L}$ 
	be an $\hat{\E} [| \cdot |]$-Cauchy sequence, that is,
	\begin{align}
		\E [| X_n - X_m | D_T] \to 0 
		\qd (n,m \to \infty).
	\end{align}
	This implies $\{ X_n D_T \}_{n=1}^{\infty} \subset \calL^1_G (\Omega)$ 
	is an $\E[| \cdot |]$-Cauchy sequence.
	Hence, from completeness of $\calL^1_G (\Omega)$, 
	there exists a unique $Y \in \calL^1_G (\Omega)$ such that
	\begin{align}
		\E [|X_n D_T - Y|] \to 0 
		\qd (n \to \infty ) .
	\end{align}
	As $X := Y D_T^{-1}$ is in $\mathscr{L}$, 
	we get $\hat{\E} [|X_n -X|] \to 0 \qd (n \to \infty )$.
	Therefore $\mathscr{L}$ is complete under the norm $\hat{\E} [| \cdot |]$.
\end{rem}

\subsection{$\bm{G}$-Novikov's condition}

For $h \in (M^2_G ( \Omega ))^d$, consider the process $D$ defined by \eqref{D}.
In this subsection, we give a sufficient condition for $D$ 
to be a symmetric $G$-martingale, which reads as follows:
there exists $\ve >0$ such that
\begin{align}\label{GN}
	\E 
	\left[ 
	\exp 
	\left( 
	\frac{1}{2} (1+ \ve ) \int_0^T h_s \cdot (d \qv{B}_s \, h_s) 
	\right) 
	\right] 
	< \infty.
\end{align}
This condition may be regarded as a sublinear counterpart 
to the well-known Novikov's condition in the classical stochastic analysis,
and we refer to it as $G$-Novikov's condition.
We remark that in the one-dimensional case, 
this condition is the same as that imposed in \cite{Xu;10}.

\begin{prop}\label{GNov,prop}
	If $h \in (M^2_G (\Omega))^d$ satisfies $G$-Novikov's condition \eqref{GN}, 
	then the process $D$ is a symmetric $G$-martingale.
\end{prop}

\begin{proof}
	Note that under the condition \eqref{GN}, 
	the usual Novikov's condition is fulfilled for all $\theta \in \cA$:
	\begin{align}
		E_{P_{\theta}} 
		\left[ 
		\exp 
		\left( 
		\frac{1}{2} \int_0^T h_s \cdot (d \qv{B}_s \, h_s) 
		\right) 
		\right] 
		< \infty .
	\end{align}
	Therefore $D$ is a $P_{\theta}$-martingale for each $\theta \in \cA$.
	In view of \pref{symG}, it remains to prove that 
	$D_t \in \calL^1_G (\Omega_t)$ for each $t \in [0,T]$.

	Fix $t \in [0,T]$ and let
	\begin{align}
		p = \frac{1+ \ve}{2 \sqrt{1 + \ve} -1} ,\qd 
		q = \frac{2 \sqrt{1 + \ve} -1}{\sqrt{1+ \ve}}.
	\end{align}
	Note that $p,q >1$ and
	\begin{align}\label{pq}
		p^2 q^2 = \frac{p q(p q-1)}{q-1} = 1+ \ve .
	\end{align}
	Then, for all $\theta \in \cA$,
	\begin{align}
		&E_{P_{\theta}} [(D_t)^p]\\
		&=  
		E_{P_{\theta}} 
		\left[ 
		\exp 
		\left( 
		\int_0^t p h_s \cdot dB_s 
		- 
		\frac{1}{2} \int_0^t 
		p^2 q h_s \cdot (d \qv{B}_s \, h_s) 
		\right) 
		\right. \\
		&\hspace{140pt}
		\times 
		\left.
		\exp 
		\left( 
		\frac{p(pq-1)}{2} \int_0^t h_s \cdot (d \qv{B}_s \, h_s)
		\right) 
		\right] \\
		&\le 
		E_{P_{\theta}} 
		\left[ 
		\exp 
		\left( 
		\int_0^t p q h_s \cdot dB_s
		- \frac{1}{2} \int_0^t 
		p^2 q^2 h_s \cdot (d \qv{B}_s \, h_s) 
		\right) 
		\right] ^{1/q}\\
		&\hspace{140pt}
		\times 
		E_{P_{\theta}} 
		\left[ 
		\exp 
		\left( 
		\frac{p q(p q-1)}{2(q-1)} \int_0^t h_s \cdot (d \qv{B}_s \, h_s)
		\right) 
		\right] ^{1-1/q} .
	\end{align}
	By \eqref{GN} and  \eqref{pq}, we have
	\begin{align}
		\ol{\E} 
		\left[ 
		\exp 
		\left( 
		\frac{p q(p q-1)}{2(q-1)} \int_0^t h_s \cdot (d \qv{B}_s \, h_s)
		\right) 
		\right] 
		< \infty,
	\end{align}
		and Novikov's condition implies that the process
	\begin{align}
		\left\{ 
		\exp 
		\left( 
		\int_0^t pq h_s \cdot dB_s 
		- 
		\frac{1}{2} \int_0^t 
		p^2 q^2 h_s \cdot (d \qv{B}_s \, h_s) 
		\right) 
		; 0 \le t \le T 
		\right\}
	\end{align}
	is a $P_{\theta}$-martingale.
	Therefore $\ol{\E} [(D_t)^p] < \infty$, and hence
	\begin{align}
		\lim_{N \to \infty} 
		\ol{\E} [D_t \one_{\{ D_t > N \}}] 
		=0.
	\end{align}
	Moreover $D_t$ has a q.c.\ version and belongs to 
	$L^0 (\Omega_t)$ since $\int_0^t h_s \cdot dB_s$
	and $\int_0^t h_s \cdot (d \qv{B}_s \, h_s)$ do by their definitions.
	Therefore, by \tref{chara}, 
	we have $D_t \in \calL^1_G (\Omega_t)$.
\end{proof}

\subsection*{Acknowledgments}

The author would like to thank Professor Yuu Hariya 
for his constant encouragement during this work.
She would also like to thank the associate editor 
for valuable comments and suggestions on an earlier version of the paper.


\end{document}